\documentclass[11pt]{article}

\usepackage{amsmath, amsthm,enumerate,cite,color}
\usepackage{etoolbox}
\providetoggle{long}
\usepackage[none]{hyphenat}
\settoggle{long}{true}
\usepackage{amssymb,graphicx}
\usepackage[numbers]{natbib}
\usepackage{cleveref}[2011/12/24]
\usepackage[plain]{fullpage}
\usepackage[text={17cm,24cm}]{geometry}
\usepackage{etoolbox}
\newtheorem{theorem}{Theorem}[section]
\newtheorem{definition}[theorem]{Definition}
\newtheorem{corollary}[theorem]{Corollary}

\newtheorem{observation}[theorem]{Observation}

\newtheorem{example}{Example}[section]

\newtheorem{proposition}[theorem]{Proposition}
\newtheorem*{propositionEquilizable}{Theorem~\ref{thm:equilizable}}

\newtheorem{lemma}[theorem]{Lemma}
\newtheorem{remark}{Remark}

\newcommand{\bbone}{{\boldsymbol{1}}}
\newcommand{\bbzero}{{\boldsymbol{0}}}

\DeclareMathOperator{\Sp}{Sp}

\title{Decomposing $1$-Sperner hypergraphs}

\author{Endre Boros\\
\small MSIS Department and RUTCOR, Rutgers University, New Jersey, USA\\
\small 100 Rockafeller Rd, Piscataway NJ 08854, USA\\
\small \texttt{Endre.Boros@rutgers.edu}\\
\and
Vladimir Gurvich\\
\small National Research University: Higher School of Economics, Moscow, Russia\\
\small \texttt{vgurvich@hse.ru}\\
\and
Martin Milani\v c\\
\small University of Primorska, UP IAM, Muzejski trg 2, SI6000 Koper, Slovenia\\
\small University of Primorska, UP FAMNIT, Glagolja\v ska 8, SI6000 Koper, Slovenia\\
\small \texttt{martin.milanic@upr.si}}

\date{\today}

\begin{document}

\maketitle

\begin{abstract}
A hypergraph is \emph{Sperner} if no hyperedge contains another one.
A Sperner hypergraph is \emph{equilizable} (resp., \emph{threshold}) if the characteristic vectors of its hyperedges are the (minimal) binary solutions to a linear equation (resp., inequality) with positive coefficients.
These combinatorial notions have many applications and are motivated by the theory of Boolean functions and integer programming.
We introduce in this paper the class of \emph{$1$-Sperner} hypergraphs, defined by the property that for every two hyperedges the smallest of their two set differences is of size one. We characterize this class of Sperner hypergraphs by a decomposition theorem and derive several consequences from it. In particular, we obtain bounds on the size of $1$-Sperner hypergraphs and their transversal hypergraphs, show that the characteristic vectors of the hyperedges are linearly independent over the reals, and prove that $1$-Sperner hypergraphs are both threshold and equilizable. The study of $1$-Sperner hypergraphs is motivated also by their applications in graph theory, which we present in a companion paper.
\end{abstract}

\section{Introduction}

In this paper we consider various classes of hypergraphs, with a focus on the
newly introduced class of $1$-Sperner hypergraphs. As we will see, this is an interesting and useful notion with many surprising properties, including a simple recursive structure. Before we explain and motivate our study and results, we  overview the necessary background definitions.

\subsection{Background}\label{sec:background}

A \emph{hypergraph} ${\cal H}$ is a pair $(V, E)$ where $V = V({\cal H})$ is a finite set of \emph{vertices} and $E = E({\cal H})$ is a set of subsets of $V$, called \emph{hyperedges}~\cite{MR1013569}. Given a positive integer $k$, a hypergraph $\mathcal{H}$ is said to be \emph{$k$-uniform} if $|e| = k$ for all $e\in E(\mathcal{H})$, and
\emph{uniform} if it is $k$-uniform for some $k$.
In particular, the (finite, simple, and undirected) graphs are precisely the $2$-uniform hypergraphs.
Four properties of hypergraphs will be particularly relevant for our study: Sperner, threshold, equilizable, and dually Sperner hypergraphs.

\bigskip
\noindent{\bf Sperner hypergraphs.} A hypergraph is said to be \emph{Sperner} if no hyperedge contains another one, that is, if $e,f\in E$ and $e\subseteq f$ implies $e = f$; see, e.g.,~\citet{MR1544925,MR0258642,MR0396277}. Sperner hypergraphs were studied in the literature under different names
including \emph{simple hypergraphs} by~\citet{MR1013569},
\emph{clutters} by~Billera~\cite{MR0307924,MR0307923} and by Edmonds and Fulkerson~\cite{MR0269433,MR0255235}, and \emph{coalitions} in the game theory literature~\cite{MR0219323}. See also~\cite{MR719998} for additional references on applications of Sperner hypergraphs in other areas of mathematics.

\bigskip
\noindent{\bf Threshold hypergraphs.} A hypergraph ${\cal H}= (V, E)$ is said to be \emph{threshold} if there exist a non-negative integer weight function $w:V\to \mathbb{Z}_{\ge 0}$ and a non-negative integer threshold $t\in \mathbb{Z}_{\ge 0}$ such that for every subset $X\subseteq V$, we have
$w(X):= \sum_{x\in X}w(x)\ge t$ if and only if $e\subseteq X$ for some $e\in E$. A pair $(w,t)$ as above will be referred to as a \emph{threshold separator} of $\mathcal{H}$.
The mapping that takes every hyperedge $e\in E$ to its \emph{characteristic vector} $\chi^e\in \{0,1\}^V$, defined by
\[
\chi^e_v = \left\{
\begin{array}{ll}
1, & \hbox{if $v\in e$} \\
0, & \hbox{otherwise\,,}
\end{array}
\right.
\]
shows that the sets of hyperedges of threshold Sperner hypergraphs are in a one-to-one correspondence with the sets of minimal feasible binary solutions of the linear inequality $w^\top x\ge t$. A set of vertices $X\subseteq V$ in a hypergraph is said to be \emph{independent} if it does not contain any hyperedge, and \emph{dependent} otherwise. Thus, threshold hypergraphs are exactly the hypergraphs admitting a linear function on the vertices separating the characteristic vectors of the independent sets from the characteristic vectors of dependent sets.

Threshold hypergraphs were defined in the uniform case by Golumbic~\cite{MR562306} and studied further by Reiterman et al.~\cite{MR791660}. The $2$-uniform threshold hypergraphs are precisely the threshold graphs, introduced by~\citet{MR0479384} and studied afterwards in numerous papers; see also the monograph by~\citet{MR1417258}. In their full generality (that is, without the restriction that the hypergraph is uniform), the concept of threshold hypergraphs is equivalent to that of threshold monotone Boolean functions. Threshold Boolean functions provide a simple but fundamental model for many questions investigated in a variety of areas including electrical engineering, artificial intelligence, game theory, cryptography, and many others; see, e.g.,~\citet{MR2184067}, \citet{MR2723262,MR2742439}, and~\citet{MR0439441}.

\begin{sloppypar}
Close interrelations between hypergraphs and monotone Boolean functions are often useful in the study of threshold and other hypergraphs, allowing for the transfer and applications of results from the theory of Boolean functions; see~\cite{MR2742439}. For example, a polynomial-time recognition algorithm for threshold monotone Boolean functions represented by their complete DNF was given by Peled and Simeone~\cite{MR798011}. The algorithm is based on linear programming and implies a polynomial-time recognition algorithm for threshold hypergraphs. To the best of our knowledge, no `purely combinatorial' polynomial-time recognition algorithm for threshold hypergraphs is known~\cite{MR2742439}.\footnote{See \citet{Smaus} for an attempt.}
\end{sloppypar}

\bigskip
\noindent{\bf Equilizable hypergraphs.} Replacing a linear inequality with positive coefficients by a linear equation maps the notion of threshold hypegraphs to the notion of equilizable hypergraphs.
A hypergraph $\mathcal{H} = (V,E)$ is said to be \emph{equilizable} if there exist a (strictly) positive integer weight function $w:V\to \mathbb{Z}_{>0}$ and a non-negative integer threshold $t\in \mathbb{Z}_{\ge 0}$ such that for every subset $X\subseteq V$, we have $w(X) = t$ if and only if $X\in E$.

Depending on the context, one may want to relax the assumption that all the weights are strictly positive to allow zero weights. However, we find the assumption of strictly positive weights useful for our study; in particular, it implies the following.

\begin{proposition}\label{obs:equilizable-Sperner}
Every equilizable hypergraph is Sperner.
\end{proposition}

\begin{proof}
Let $\mathcal{H} = (V,E)$ be an equilizable hypergraph and let
$w:V\to \mathbb{Z}_{>0}$ and $t\in \mathbb{Z}_{\ge 0}$ be such that for every subset $X\subseteq V$, we have $w(X) = t$ if and only if $X\in E$.
Suppose for a contradiction that $\mathcal{H}$ is not Sperner.
Then, there exist two hyperedges $e$ and $f$ of  $\mathcal{H}$ such that
$e\subset f$. Consequently, $w(e) = t = w(f)$, which implies that
$w(f\setminus e) = 0$, contrary to the fact that $w$ is strictly positive.
\end{proof}

Equilizable hypergraphs are a very natural family. The sets of hyperedges of an equilizable hypergraph are in a one-to-one correspondence with the sets of binary solutions to a linear equality of the form $w^\top x= t$ where $w\in \mathbb{Z}_{>0}^V$ and $t\in \mathbb{Z}_{\ge 0}$, that is, with the sets of binary vectors that a single hyperplane with positive coefficients can cut out from the hypercube. It is thus not surprising that properties of equilizable hypergraphs are fundamental in integer (or binary) programming. In particular, an old result of Mathews~\cite{MR1576653} shows how to reduce two linear Diophantine equations with strictly positive coefficients within an arbitrary set to a single equivalent linear equation of the same type. Motivated by integer programming considerations, many authors generalized Mathews' result in a variety of ways, see~\citet{MR0373595},~\citet{MR0303937},~\citet{EW},~\citet{MR0406500},~\citet{MR0314456}, and~\citet{MR0386688}.\footnote{In some sense, these equation aggregation results are not unexpected. The intersection of two hyperplanes in $\mathbb{R}^n$ is an $(n-2)$-dimensional subspace ${\cal F}$. Any integer point not included in ${\cal F}$ extends ${\cal F}$ to a unique hyperplane (single equality). Since in a bounded region there are only finitely many such feasible integer points, there are only finitely many hyperplanes through ${\cal F}$ that contain an integer point not contained in ${\cal F}$. Thus we must have infinitely many hyperplanes containing ${\cal F}$ that do not contain any other integer feasible point. The only nontrivial part is the numerical construction of an explicit hyperplane. All constructions from~\cite{MR0303937,MR0373595,MR0314456,EW,MR0386688,MR1576653,MR0406500} end up introducing exponentially growing coefficients.}
Furthermore, Mathews' result implies that the class of equilizable hypergraphs on a given vertex set is closed under intersection.\footnote{The intersection of two hypergraphs $\mathcal{H}_1 = (V_1,E_1)$ and $\mathcal{H}_2 = (V_2,E_2)$ is defined in the natural way, namely as the hypergraph $(V_1\cap V_2,E_1\cap E_2)$.}
A related question about linear inequalities led~\citet{MR0479384} to the introduction of threshold graphs.

Equilizable hypergraphs can also be seen as a generalization of the class of equistable graphs, defined as follows. A \emph{stable set} (or: \emph{independent set}) in a graph $G$ is a set of pairwise non-adjacent vertices.
A graph $G = (V,E)$ is said to be \emph{equistable} if there exist a  (strictly) positive integer weight function $w:V\to \mathbb{Z}_{>0}$ and a non-negative integer threshold $t$ such that for every subset $X\subseteq V$, we have $w(X) = t$ if and only if $X$ is an (inclusion-)maximal stable set of $G$. Equistable graphs were introduced in 1980 by~\citet{MR553649}, who proved that every threshold graph is equistable. While equistable graphs were originally defined using a function $\varphi:V\to \mathbb{R}_{\ge 0}$ such that for every subset $X\subseteq V$, we have $\varphi(X) = 1$ if and only if $X$ is a maximal stable set of $G$, it is not difficult to see that the above two definitions are equivalent. Equistable graphs were studied in a series of papers~\cite{MR3623393,Korach-Peled-Rotics,MPS13,MR3474710,MM11,MR3575013,Pel_Rot,Lev_Mil}. However, unlike threshold graphs, the structure of equistable graphs is not understood and the complexity of the problem of recognizing equistable graphs is open. The connection between equistable graphs and equilizable hypergraphs can be easily explained using the notion of stable set hypergraphs. The \emph{stable set hypergraph} of a graph $G$ is the hypergraph $\mathcal{S}(G)$ with vertex set $V(G)$ and in which the hyperedges are exactly the maximal stable sets of $G$. Clearly, a graph $G$ is equistable if and only if its stable set hypergraph is equilizable.

\bigskip
\noindent{\bf Dually Sperner hypergraphs.}
Sperner hypergraphs can be equivalently defined as the hypergraphs such that every two distinct hyperedges $e$ and $f$ satisfy
\begin{equation}\label{eq1}
\min\{|e\setminus f|,|f\setminus e|\}\ge 1\,.
\end{equation}
This observation motivated Chiarelli and Milani\v c to call in~\cite{MR3281177} a hypergraph ${\cal H}$ \emph{dually Sperner} if
every two distinct hyperedges $e$ and $f$ satisfy
\begin{equation*}
\min\{|e\setminus f|,|f\setminus e|\}\le 1\,.
\end{equation*}
The following result was shown in~\cite{MR3281177}.

\begin{theorem}[Chiarelli-Milani\v c~\cite{MR3281177}]\label{thm:dually-Sperner}
Every dually Sperner hypergraph is threshold.
\end{theorem}


\subsection{The main definition}

The main notion studied in this paper is given by the following.

\begin{definition}
Given a positive integer $k$, we say that a hypergraph $\mathcal{H}$ is \emph{$k$-Sperner} if every two distinct hyperedges $e$ and $f$ satisfy $$1\le \min\{|e\setminus f|,|f\setminus e|\}\le k\,.$$
In particular, $\mathcal{H}$ is \emph{$1$-Sperner} if every two distinct hyperedges $e$ and $f$ satisfy $$\min\{|e\setminus f|,|f\setminus e|\}= 1\,,$$
or, equivalently, if, for any two distinct hyperedges $e$ and $f$ of $\mathcal{H}$ with $|e| \leq |f|$, we have $|e\setminus f| = 1$.
\end{definition}

Denoting by $\mathcal{S}_k$ the class of all $k$-Sperner hypergraphs and by $\mathcal{S}$ the class of all Sperner hypergraphs,
it is clear that these families of hypergraphs are related by
the following chain of inclusions
$$\mathcal{S}_1\subseteq \mathcal{S}_2 \subseteq \ldots \subseteq \bigcup_{k\ge 1}\mathcal{S}_k = \mathcal{S}\,.$$
The inclusions follow immediately from the definitions, while
the equality follows from~\eqref{eq1}.
Moreover, since we do not allow multiple hyperedges, every $2$-uniform hypergraph (that is, a graph) is $k$-Sperner for every $k\ge 2$.
We should therefore not expect useful decomposition properties for the classes of $k$-Sperner hypergraphs for $k\ge 2$. We focus in this paper on the case $k = 1$ and show that hypergraphs in the corresponding subfamily $\mathcal{S}_1$ have a nice structure. Note that by definition, all hypergraphs with at most one hyperedge (possibly with no vertices) are $1$-Sperner. Note also that
a hypergraph is $1$-Sperner if and only if it is both Sperner and dually Sperner.

\begin{sloppypar}
The concept of $1$-Sperner hypergraphs already appeared in some graph theoretical research. \citet{MR3281177,ISAIM2014} made use of dually Sperner hypergraphs to characterize two classes of graphs defined by the following properties: every induced subgraph has a non-negative linear vertex weight function separating the characteristic vectors of all total dominating sets~\cite{ISAIM2014}, resp.~connected dominating sets~\cite{MR3281177}, from the characteristic vectors of all other sets. Due to the close relation between $1$-Sperner and dually Sperner hypergraphs (see Observation~\ref{obs:Sp-red}), all the results from~\cite{MR3281177,ISAIM2014} can be equivalently stated using $1$-Sperner hypergraphs. In particular, the results of the extended abstract~\cite{ISAIM2014} are stated using the $1$-Sperner property in the full version of the paper~\cite{CM-ISAIM2014}.
\end{sloppypar}

\subsection{Our results}

Our main result is a decomposition theorem for $1$-Sperner hypergraphs.
We also derive several consequences of it.

\bigskip
\noindent{\bf The decomposition theorem.} We define a simple operation on hypergraphs called gluing and show that it produces (with only one small exception) a new $1$-Sperner hypergraph from a given pair of $1$-Sperner hypergraphs; see Fig.~\ref{fig:1} for an example illustrated with incidence matrices. Conversely, we show that every $1$-Sperner hypergraph with at least one vertex is the gluing of two smaller $1$-Sperner hypergraphs; see Theorem~\ref{thm:decomposition}.

\bigskip
\noindent{\bf Consequences.} We use the decomposition theorem to prove the following properties of $1$-Sperner hypergraphs.
\begin{enumerate}[a)]
\item Every $1$-Sperner hypergraph is threshold and has a positive threshold separator; see Theorem~\ref{thm:1-Sperner-threshold}. In particular, this gives a new, constructive proof of the fact that every dually Sperner hypergraph is threshold, obtained first by~\citet{MR3281177}; see Theorem~\ref{thm:dually-Sperner}.

\item Every $1$-Sperner hypergraph is equilizable; see Theorem~\ref{thm:equilizable}.

\item The characteristic vectors of the hyperedges of a $1$-Sperner hypergraph are linearly independent over the reals; see Theorem~\ref{prop:linear-independence}.
    This implies that the number of hyperedges cannot exceed the number of vertices, thus giving a sharp upper bound on the size of a $1$-Sperner hypergraph in terms of its order; see~\Cref{cor:edges}.
    We also give a sharp lower bound on the size of a $1$-Sperner hypergraph without universal, isolated, and twin vertices, in terms of its order; see~\Cref{prop:lower-bound}.

\item The number of minimal transversals of a $1$-Sperner hypergraph is bounded from above by a quadratic function of its order and they can be efficiently generated; see Theorem~\ref{thm:min-transversal-of-1-Sperner}.
\end{enumerate}

Our study of $1$-Sperner hypergraphs is motivated not only by their nice combinatorial properties but also by their numerous applications in graph theory. Some of them were already mentioned above and we obtained several others. To keep the length of this paper reasonable, we decided to present those results in a separate paper~\cite{BGM-1-Sperner-graphs}. We briefly summarize them here.

\begin{sloppypar}
We use the characterizations of so-called threshold and domishold graphs in terms of forbidden induced subgraphs due to~\citet{MR0479384} and~\citet{MR0491342}, respectively, to derive further characterizations of these graph classes in terms of $1$-Spernerness, thresholdness, and $2$-asummability properties of several related hypergraphs, namely their vertex cover, clique, independent set, dominating set, and closed neighborhood hypergraphs.
\end{sloppypar}

Furthermore, we use the decomposition theorem for $1$-Sperner hypergraphs (Theorem~\ref{thm:decomposition}) to derive decomposition theorems for four classes of graphs, namely two classes of split graphs, a class of bipartite graphs, and a class of cobipartite graphs. These decomposition theorems are based on certain matrix partitions of the corresponding graphs and
give rise to new classes of graphs of bounded clique-width and to new polynomially solvable cases of variants of domination.

\subsection{Interrelations between the considered classes of hypergraphs}

In Fig.~\ref{fig:Hasse}, we show the Hasse diagram of the partial order of the hypergraph classes studied in this paper, ordered with respect to inclusion.

\begin{figure}[h!]
  \centering
   \includegraphics[width=0.87\textwidth]{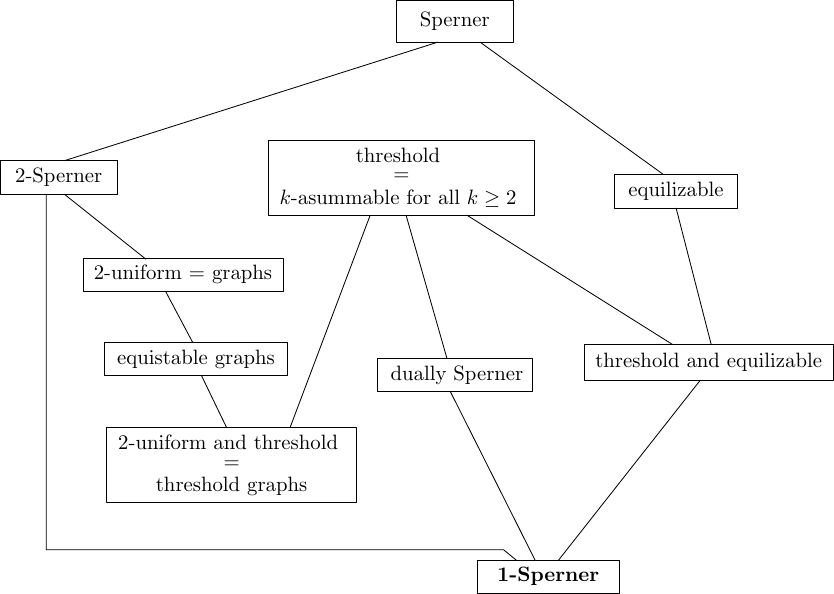}
  \caption{Inclusion relations between several classes of hypergraphs.}\label{fig:Hasse}
\end{figure}

\begin{sloppypar}
The fact that every $1$-Sperner hypergraph is threshold and equilizable
is proved in Theorems~\ref{thm:1-Sperner-threshold} and~\ref{thm:equilizable}, respectively.
The fact that every dually Sperner hypergraph is threshold was proved by~\citet{MR3281177}. The fact that every threshold graph is equistable was proved by~\citet{MR553649}. The fact that every equilizable hypergraph is Sperner was proved in~\Cref{obs:equilizable-Sperner}. The remaining inclusions are trivial.
\end{sloppypar}

Finally, the following examples show that all inclusions are strict and there are no other inclusions:
\begin{itemize}
  \item the complete graph $K_4$ is a $2$-uniform hypergraph that is threshold but not dually Sperner;
  \item the hypergraph with vertex set $\{1,2,3\}$ and hyperedge set $\{\{1,2,3\}\}$ is $1$-Sperner but not $2$-uniform;
  \item the hypergraph with vertex set $\{1\}$ and hyperedge set $\{\emptyset,\{1\}\}$ is dually Sperner but not Sperner,
 \item an equilizable hypergraph that is not threshold is presented in Example~\ref{example-equilizable-not-threshold};
\item a $2$-uniform threshold hypergraph that is not equilizable is presented in Example~\ref{example-threshold-not-equilizable};
\item a threshold and equilizable hypergraph that is neither dually Sperner nor $2$-Sperner is the complete $3$-uniform hypergraph $\mathcal{H}_{6,3}$; see Example~\ref{example-threshold-equilizable-not-1-Sperner};
 \item the cycle $C_4$ is an equistable graph that, when viewed as a $2$-uniform hypergraph, it is not threshold~\cite{MR0479384};
\item the path $P_4$ is a graph that is not equistable~\cite{MR553649}; moreover, it is also a Sperner hypergraph that is not equilizable.
\end{itemize}
The remaining non-inclusions follow by transitivity.

\medskip
\begin{sloppypar}
\noindent{\bf Structure of the paper.}
In \Cref{sec:preliminaries}, we collect the necessary definitions and preliminary results. We also consider several operations on hypergraphs and show that the class of $1$-Sperner hypergraphs is (almost always) closed under these operations. In Section~\ref{sec:uniform} we give a necessary condition for a uniform hypergraph to be $1$-Sperner and identify two families of uniform $1$-Sperner hypergraphs.
Building on the results of Sections~\ref{sec:preliminaries} and~\ref{sec:uniform}, we develop in Section~\ref{sec:decomposition} the composition theorem for $1$-Sperner hypergraphs. Various consequences of this theorem are examined in Section~\ref{sec:applications}.
\end{sloppypar}

\section{Definitions and hypergraph operations}\label{sec:preliminaries}

The \emph{order} and the \emph{size} of a hypergraph $\mathcal{H}$ refer to the number of its vertices, resp.~hyperedges.
Every hypergraph $\mathcal{H}=(V,E)$ with a fixed pair of orderings
of its vertices and edges, say
$V = \{v_1,\ldots,  v_n\}$,
and
$E = \{e_1,\ldots,  e_m\}$,
can be represented with its \emph{incidence matrix} $A^\mathcal{H}\in \{0,1\}^{E\times V}$ having rows and columns indexed by edges and vertices of $\mathcal{H}$, respectively,
and defined
as
\[
A^\mathcal{H}_{i,j} =
\left\{
\begin{array}{ll}
1, & \hbox{if $v_j\in e_i$;} \\
0, & \hbox{otherwise.}
\end{array}
\right.
\]
Note the slight abuse of notation above: the incidence matrix does not depend only on the hypergraph but also on the
pair of orderings of its vertices and edges. We will be able to neglect this technical issue often in the paper, but not always.
We will therefore say that two matrices $A$ and $B$ of the same dimensions are \emph{permutation equivalent}, and
denote this fact by $A\cong B$, if $A$ can be obtained from $B$ by permuting some of its rows and/or columns.
For later use, we state a simple property of incidence matrices of a hypergraph.

\begin{remark}\label{obs:matrices}
Let $\mathcal{H}=(V,E)$ be a hypergraph with a fixed pair of orderings
of its vertices and edges, respectively, and let $A^\mathcal{H}$ be the corresponding incidence matrix.
Then, any permutation of the vertices and/or edges of $\mathcal{H}$
results in an incidence matrix that is
permutation equivalent to $A^\mathcal{H}$.
Moreover, any matrix that is permutation equivalent to $A^\mathcal{H}$
is the incidence matrix of $\mathcal{H}$ with respect to some pair
of orderings of its vertices and edges.
\end{remark}

\medskip
\begin{sloppypar}
\noindent{\bf $k$-asummable hypergraphs.} A hypergraph is \emph{$k$-asummable} if it has no $k$ (not necessarily distinct) independent sets $A_1,\ldots, A_k$ and $k$ (not necessarily distinct) dependent sets $B_1,\ldots, B_k$ such that $$\sum_{i = 1}^k\chi^{A_i} = \sum_{i = 1}^k\chi^{B_i}\,.$$ A hypergraph is \emph{asummable} if it is $k$-asummable for every $k\ge 2$. The following characterization of threshold graphs follows from analogous characterizations of threshold monotone Boolean functions; see~\cite{MR2742439}.

\begin{theorem}[Chow~\cite{Chow} and Elgot~\cite{5397278}]\label{thm:ChowElgot}
A hypergraph is threshold if and only if it is asummable.
\end{theorem}

\medskip
Next we consider several operations on hypergraphs and
show that the class of $1$-Sperner hypergraphs is (almost always) closed under these operations.
\end{sloppypar}

\subsection{Hypergraph complementation}

Given a hypergraph $\mathcal{H}=(V,E)$, the \emph{complement} of $\mathcal{H}$ is the hypergraph $\overline{\mathcal{H}}$ with $V(\overline{\mathcal{H}}) = V$ and $E(\overline{\mathcal{H}}) = \{\overline e\mid e\in E(\mathcal{H})\}$,
fwhere $\bar e $ denotes  $V \setminus e$ for any subset $e\subseteq V$.

\begin{proposition}\label{prop:complement}
The complement of every $1$-Sperner hypergraph is $1$-Sperner.
\end{proposition}

\begin{proof}
This follows directly from the definition, using the fact that for every two sets $e,f\subseteq V$, we have
$e\setminus f = \overline f\setminus \overline e$ and
$f\setminus e = \overline e\setminus \overline f$.
\end{proof}

As the next example shows, the closure under complementation does not hold for the classes of threshold Sperner hypergraphs and $2$-asummable Sperner hypergraphs.

\begin{example}
Consider the $3$-uniform hypergraph $\mathcal{H} = (V,E)$ with $V = \{1,\ldots,6\}$ in which a set $e=\{x,y,z\}\subseteq V$ forms a hyperedge if and only if $e$ contains at least two elements of $\{1,2,3,4\}$.
Then $\mathcal{H}$ is a threshold hypergraph, with a threshold separator $(w,t)$ given by $(w(1),\ldots, w(6)) = (3,3,3,3,1,1)$ and $t = 7$. Since $\mathcal{H}$ is threshold, it is also $2$-asummable by Theorem~\ref{thm:ChowElgot}. Its complement is the hypergraph $\overline{\mathcal{H}} = (V,\overline{E})$ with $\overline{E} = \{e\subseteq V :|e| = 3, e\nsubseteq \{1,2,3,4\}\}$. Since in $\overline{\mathcal{H}}$, sets $A_1 = \{1,2,3,4\}$ and $A_2 = \{5,6\}$ are independent, while sets $B_1 = \{1,2,5\}$ and $B_2 = \{3,4,6\}$ are hyperedges, such that $\chi^{A_1} +\chi^{A_2} = \chi^{B_1}+\chi^{B_2}$, we infer that
$\overline{\mathcal{H}}$ is not $2$-asummable, hence also not threshold.
\end{example}

\subsection{Gluing of hypergraphs}

The decomposition theorem (Theorem~\ref{thm:decomposition}) is based on the following general operation.

\begin{definition}[Gluing of two hypergraphs]
Given a pair of vertex-disjoint hypergraphs
$\mathcal{H}_1 = (V_1,E_1)$ and
$\mathcal{H}_2 = (V_2,E_2)$ and a new vertex $z\not\in V_1\cup V_2$,
the \emph{gluing of $\mathcal{H}_1$ and $\mathcal{H}_2$} is the hypergraph
$\mathcal{H} = \mathcal{H}_1\odot  \mathcal{H}_2$ such that
$$V(\mathcal{H}) = V_1\cup V_2\cup\{z\}$$ and
$$E(\mathcal{H}) = \{\{z\}\cup e\mid e\in E_1\} \cup \{V_1\cup e\mid e\in E_2\}\,.$$
\end{definition}

Let us note the operation of gluing is well-defined also if some of the sets $V_1$, $V_2$, $E_1$, and $E_2$ are empty. The operation can be visualized easily in terms of incidence matrices.
Let $n_i = |V_i|$ and $m_i = |E_i|$ for $i = 1,2$, and let us denote
by $\bbzero^{k,\ell}$, resp.~$\bbone^{k,\ell}$, the $k\times \ell$ matrix of all zeroes, resp.~of all ones.
Then, the incidence matrix of the gluing of $\mathcal{H}_1$ and $\mathcal{H}_2$ can be written as
\[
A^{\mathcal{H}_1\odot \mathcal{H}_2} =\left(
    \begin{array}{ccc}
      \bbone^{m_1,1} & A^{\mathcal{H}_1} & \bbzero^{m_1, n_2} \\
      \bbzero^{m_2,1} & \bbone^{m_2, n_1}  & A^{\mathcal{H}_2} \\
    \end{array}
  \right)\,.
\]
See Fig.~\ref{fig:1} for an example.

\begin{figure}[h!]
  \centering
   \includegraphics[width=150mm]{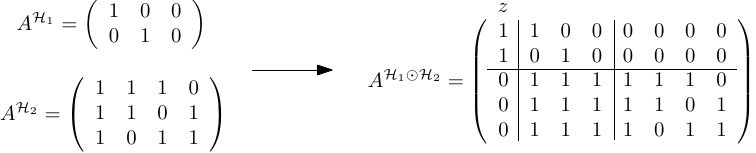}
  \caption{An example of gluing of two hypergraphs.}\label{fig:1}
\end{figure}

To further illustrate the operation of gluing, let us note that the operation generalizes the operations of adding an isolated or a universal vertex.
A vertex $u$ in a hypergraph $\mathcal{H} = (V,E)$ is said to be \emph{universal} (resp., \emph{isolated}) if it is contained in all (resp., in no) hyperedges. The operations of adding an isolated or a universal vertex to a hypergraph are defined in the natural way.

\begin{observation}\label{obs:iso-uni}
For every hypergraph $\mathcal{H}$, the following holds:
\begin{itemize}
  \item The hypergraph obtained by adding an isolated vertex to $\mathcal{H}$ is the result of gluing of $(\emptyset,\emptyset)$ and $\mathcal{H}$.
  \item The hypergraph obtained by adding a universal vertex to $\mathcal{H}$ is the result of gluing of $\mathcal{H}$ and $(\emptyset,\emptyset)$.
\end{itemize}
\end{observation}

The gluing and the complementation operations are related as follows (the proof is left as an easy exercise for the reader):

\begin{observation}\label{obs:complement}
If $\mathcal{H} = \mathcal{H}_1\odot \mathcal{H}_2$, then $\overline{\mathcal{H}} = \overline{\mathcal{H}_2}\odot \overline{\mathcal{H}_1}$
(assuming that in both gluing operations the same new vertex is used).
\end{observation}

It is easy to see that if the gluing of $\mathcal{H}_1$ and $\mathcal{H}_2$ is a $1$-Sperner hypergraph, then both constituent hypergraphs
are $1$-Sperner. We record this fact for later use.

\begin{observation}\label{obs:constituents}
If the gluing of $\mathcal{H}_1$ and $\mathcal{H}_2$ is a $1$-Sperner hypergraph, then $\mathcal{H}_1$ and $\mathcal{H}_2$ are also $1$-Sperner.
\end{observation}

The next proposition establishes a partial converse. Gluing preserves $1$-Spernerness, unless the resulting hypergraph is not Sperner, which happens only in one very special case.

\begin{proposition}\label{prop:gluing}
For every pair
$\mathcal{H}_1 = (V_1,E_1)$ and $\mathcal{H}_2 = (V_2,E_2)$ of
vertex-disjoint \hbox{$1$-Sperner} hypergraphs, their gluing
 $\mathcal{H}_1\odot \mathcal{H}_2$ is a $1$-Sperner hypergraph,
unless
$E_1 = \{V_1\}$ and $E_2 =\{\emptyset\}$
(in which case the
hypergraph  $\mathcal{H}_1\odot  \mathcal{H}_2$ is not Sperner).
\end{proposition}

\begin{proof}
Let $e$ and $f$ be two distinct edges of $\mathcal{H}_1\odot  \mathcal{H}_2$.
If $z\in e\cap f$ then their differences are the same as the corresponding differences of
$e\setminus\{z\}$ and $f\setminus\{z\}$, both of which are hyperedges of $\mathcal{H}_1$.
If $z\not\in e\cup f$ then
their differences are the same as the corresponding differences of
$e\setminus V_1$ and $f\setminus V_1$, both of which are hyperedges of $\mathcal{H}_2$.
If $z\in e\setminus f$, then $e\setminus f = \{z\}$ and $f\setminus e\neq \emptyset$,
unless $e = V_1$, and $f = \emptyset$ (which implies $E_1=\{V_1\}$ and $E_2=\{\emptyset\}$
by our assumption that both $\mathcal{H}_1$ and $\mathcal{H}_2$ are $1$-Sperner).
The case of $z\in f\setminus e$ is symmetric.
\end{proof}

\subsection{Sperner reductions and hypergraph transversals}

Given a hypergraph $\mathcal{H} = (V,E)$, its \emph{Sperner reduction}, $\Sp(\mathcal{H})$, is the hypergraph with vertex set $V$ and with hyperedges the inclusion-minimal elements of $E$.

The following observation is easy to prove from the definitions.

\begin{observation}\label{obs:Sp-red-1}
For any two hypergraphs $\mathcal{H}_1$ and $\mathcal{H}_2$, if
$E(\Sp(\mathcal{H}_1))\subseteq E(\mathcal{H}_2)\subseteq E(\mathcal{H}_1)$, then $E(\Sp(\mathcal{H}_1)) = E(\Sp(\mathcal{H}_2))$.
\end{observation}

The following observation is a direct consequence of the definitions of dually Sperner and $1$-Sperner hypergraphs.

\begin{observation}\label{obs:Sp-red}
The Sperner reduction of every dually Sperner hypergraph is $1$-Sperner.
\end{observation}

Furthermore, the problem of studying the thresholdness property in a class of hypergraphs reduces to the class of their Sperner reductions.

\begin{proposition}\label{lem:Sperner-reduction}
Let $\mathcal{H} = (V,E)$ be a hypergraph, let $w:V\to \mathbb{Z}_{\ge 0}$ and $t\in \mathbb{Z}_{\ge 0}$. Then, $(w,t)$ is a threshold separator of $\mathcal{H}$ if and only if $(w,t)$ is a threshold separator of $\Sp(\mathcal{H})$. In particular, $\mathcal{H}$ is threshold if and only if its Sperner reduction is threshold.
\end{proposition}

\begin{proof}
Let us call a subset of vertices $X\subseteq V$ \emph{heavy} if $w(X)\ge t$, and \emph{light}, otherwise.
The pair $(w,t)$ is a threshold separator of $\mathcal{H}$ if and only if
the heavy subsets of $V$ are precisely those containing a hyperedge of $\mathcal{H}$.
Since the set of heavy subsets depends only on $(w,t)$ and not on $\mathcal{H}$ and
a subset of $V$ contains a hyperedge of $\mathcal{H}$ if and only if
it contains a hyperedge of $\Sp(\mathcal{H})$, the proposition follows.
\end{proof}

Let $\mathcal{H} = (V,E)$ be a hypergraph. A \emph{transversal} of $\mathcal{H}$ is a set of vertices intersecting all hyperedges of $\mathcal{H}$. The \emph{transversal hypergraph} $\mathcal{H}^T$ is the hypergraph with vertex set $V$ in which a set $X\subseteq V$ is a hyperedge if and only if $X$ is an inclusion-minimal transversal of $\mathcal{H}$.
(In particular, if $\mathcal{H}$ has no hyperedge, then its transversal hypergraph is $\mathcal{H}^T = (V(\mathcal{H}), \{\emptyset\})$.)

\begin{observation}[see, e.g., \citet{MR1013569}]\label{obs:transversals}
If $\mathcal{H}$ is a Sperner hypergraph, then $(\mathcal{H}^T)^T = \mathcal{H}$.
\end{observation}

A pair of a mutually transversal Sperner hypergraphs naturally corresponds to a pair of dual monotone Boolean functions, see~\cite{MR2742439}.

\medskip
The next proposition, which will be used in the proof of Theorem~\ref{thm:min-transversal-of-1-Sperner},
describes how to compute the transversal hypergraph of the gluing of two hypergraphs $\mathcal{H}_1$ and $\mathcal{H}_2$ from their transversal hypergraphs.

\begin{proposition}\label{prop:gluing-transversal}
Let $\mathcal{H}$ be a gluing of two vertex-disjoint hypergraphs
$\mathcal{H}_1 = (V_1,E_1)$ and
$\mathcal{H}_2 = (V_2,E_2)$ with
$V(\mathcal{H}) = V_1\cup V_2\cup\{z\}$.
Then,
$$E\left(\mathcal{H}^T\right) = \left\{
                       \begin{array}{ll}
                         \Sp\Big(E\left(\mathcal{H}_1^T\right)\cup \left\{\{z\}\cup e\mid e\in E\left(\mathcal{H}_2^T\right)\right\}\cup \left\{\{z,u\}\mid u\in V_1\right\}\Big), & \hbox{if $E_1\neq \emptyset$;} \\*[3mm]
                        \Sp\Big(E\left(\mathcal{H}_2^T\right) \cup \left\{\{u\}\mid u\in V_1\right\}\Big), & \hbox{if $E_1 = \emptyset$.}
                       \end{array}
                     \right.$$
\end{proposition}

\begin{proof}
Let
$$F = \left\{
                       \begin{array}{ll}
                         E\left(\mathcal{H}_1^T\right)\cup \left\{\{z\}\cup e\mid e\in E\left(\mathcal{H}_2^T\right)\right\}\cup \left\{\{z,u\}\mid u\in V_1\right\}, & \hbox{if $E_1\neq \emptyset$;} \\*[3mm]
                        E\left(\mathcal{H}_2^T\right) \cup \left\{\{u\}\mid u\in V_1\right\}, & \hbox{if $E_1 = \emptyset$.}
                       \end{array}
                     \right.$$
We will first show that every set in $F$ is a transversal of $\mathcal{H}$ and then we will argue that every minimal transversal of $\mathcal{H}$ appears in $F$. Together, by~\Cref{obs:Sp-red-1}, these two claims will imply the stated equality.

The first claim is easy to see by the definition of the gluing operation.

For the second claim, let $X$ be a minimal transversal of $\mathcal{H}$.
Suppose first that $E_1 = \emptyset$. Note that in this case $z$ is an isolated vertex of $\mathcal{H}$, so no minimal transversal of
$\mathcal{H}$ can contain $z$. If $X\cap V_1\neq \emptyset$, then $X = \{u\}$ for some $u\in V_1$ by the minimality property.
If $X\cap V_1 = \emptyset$, then $X$ must be a minimal transversal of $\mathcal{H}$.

Finally, assume that $E_1 \neq \emptyset$.
Suppose also that $V_1= \emptyset$. Then all minimal transversals of $\mathcal{H}$ must contain $z$ and must intersect all hyperedges of $\mathcal{H}_2$. Thus, $X$ must have the form $X = \{z\}cup e$ for some $e\in E\left(\mathcal{H}_2^T\right)$; in particular $X\in F$.
Now let $V_1 \neq \emptyset$.
If $z\not\in X$, then $X$ must be a minimal transversal of $\mathcal{H}_1$.
If $z\in X$ and $X\cap V_1\neq\emptyset$, then by minimality we must have
$X = \{z,u\}$ for some $u\in V_1$.
If $z\in X$ and $X\cap V_1= \emptyset$, then we  must have
$X = \{z\}\cup e$ for some $e\in E\left(\mathcal{H}_2^T\right)$.
In either case, $X$ belongs to $F$. This completes the proof.
\end{proof}

\subsection{Ungluing hypergraphs}

We next introduce some terminology related to hypergraphs that are the result of a gluing operation. Given a vertex $z$ of a hypergraph $\mathcal{H}$, we say that a hypergraph $\mathcal{H}$ is \emph{$z$-decomposable} if for every two hyperedges $e,f\in E(\mathcal{H})$ such that $z\in e\setminus f$, we have $e\setminus\{z\}\subseteq f$.
Equivalently, if the vertex set of $\mathcal{H}$ can be partitioned as $V(\mathcal{H}) = \{z\}\cup V_1\cup V_2$ such that
$\mathcal{H} = \mathcal{H}_1\odot  \mathcal{H}_2$ for some hypergraphs $\mathcal{H}_1 = (V_1,E_1)$ and $\mathcal{H}_2 = (V_2,E_2)$.
We call
$\mathcal{H} = \mathcal{H}_1\odot  \mathcal{H}_2$ a  \emph{$z$-decomposition} of  $\mathcal{H}$.

The following proposition gathers some basic properties of decomposability.

\begin{proposition}\label{prop:complement-2}
Let $\mathcal{H}$ be a hypergraph. Then, the following holds:
\begin{enumerate}[(i)]
\item If $z$ is a vertex of $\mathcal{H}$ such that $\mathcal{H}$ is $z$-decomposable,
then $\overline{\mathcal{H}}$ is also $z$-decomposable.
  \item If $z$ is an isolated or a universal vertex of $\mathcal{H}$, then
  $\mathcal{H}$ is $z$-decomposable.
\end{enumerate}
\end{proposition}

\begin{proof}
Statement $(i)$ follows from Observation~\ref{obs:complement}.

Statement $(ii)$ is related to Observation~\ref{obs:iso-uni}. Note that $z$ is universal in $\mathcal{H}$ if and only if it is isolated in $\overline{\mathcal{H}}$. By $(i)$, it therefore suffices to prove the statement for the case when $z$ is an isolated vertex of $\mathcal{H}$.
In this case, the column of $A^\mathcal{H}$ indexed by $z$ is the all zero vector.
It follows that $\mathcal{H}$ is $z$-decomposable, as follows:
$V(\mathcal{H}) = \{z\}\cup V_1\cup V_2$ with $\mathcal{H} = \mathcal{H}_1\odot  \mathcal{H}_2$,
$\mathcal{H}_1 = (V_1,E_1)$ and $\mathcal{H}_2 = (V_2,E_2)$, where $V_1 = E_1 = \emptyset$, $V_2 = V\setminus\{z\}$, and $E_2 = E(\mathcal{H})$.
\end{proof}

Recall that by~\Cref{obs:constituents}, if a $1$-Sperner hypergraph $\mathcal{H}$ has a $z$-decomposition $\mathcal{H} = \mathcal{H}_1\odot  \mathcal{H}_2$, then $\mathcal{H}_1$ and $\mathcal{H}_2$ are also $1$-Sperner.

\medskip
Whether a given hypergraph is $z$-decomposable for some vertex $z$ can be checked in a straightforward way in polynomial time.

\begin{proposition}\label{prop:complexity}
Let $\mathcal{H} = (V,E)$ be a hypergraph with $V\neq \emptyset$ and $E\neq \emptyset$ given by the lists of its vertices and hyperedges.
We can recognize if $\mathcal{H}$ is $z$-decomposable for some $z\in V$ and find a corresponding $z$-decomposition
$\mathcal{H} = \mathcal{H}_1\odot  \mathcal{H}_2$ (if there is one) in time $\mathcal{O}(|V|^2|E|)$.
\end{proposition}

\begin{proof}
It suffices to show that for a given vertex $z\in V$ we can verify in time
$\mathcal{O}(|V||E|)$  if $\mathcal{H}$ is $z$-decomposable and find a corresponding $z$-decomposition $\mathcal{H} = \mathcal{H}_1\odot  \mathcal{H}_2$ (if there is one).

First, we partition the hyperedges of $\mathcal{H}$ into those containing $z$ and those not containing $z$.
Secondly, we compute the sets $E_1 = \{e\setminus \{z\}\mid z\in e\in E\}$
and $V_1 = \cup\{e\mid e\in E_1\}$.
Thirdly, we verify if for every hyperedge $e\in E$ not containing $z$, we have $V_1\subseteq e$. If this condition is not satisfied, then $\mathcal{H}$ is not $z$-decomposable.
If the condition is satisfied, then we compute the sets
$V_2 = V\setminus (V_1\cup\{z\})$ and
$E_2 = \{e\setminus V_1\mid z\not\in e\in E\}$.
We return the $z$-decomposition $\mathcal{H} = \mathcal{H}_1\odot  \mathcal{H}_2$, where
$\mathcal{H}_1 = (V_1, E_1)$ and
$\mathcal{H}_2 = (V_2, E_2)$.
Since each of the steps can be performed in time $\mathcal{O}(\sum_{e\in E}|e|) = \mathcal{O}(|V||E|)$, the claimed time complexity follows.
\end{proof}

Clearly, if $\mathcal{H} = (V,E)$ is a hypergraph with at least one vertex and no hyperedges, then $\mathcal{H}$ is $z$-decomposable for every $z\in V$ and a $z$-decomposition of $\mathcal{H}$ can be computed in time $\mathcal{O}(|V|)$.

\section{Uniform $1$-Sperner hypergraphs}
\label{sec:uniform}

In the next lemma we give a necessary condition for a uniform hypergraph to be $1$-Sperner. The condition will be used in the proof of Theorem~\ref{thm:decomposition}.

\begin{lemma}\label{lem2}
Let $\mathcal{H}$ be a $k$-uniform $1$-Sperner hypergraph, where $k\ge 1$.
Then, either there is a subset $P$ of vertices of size $k-1$ such that $P\subseteq e$ for all $e\in E(\mathcal{H})$
or there is a subset $Q$ of vertices of size $k+1$ such that $e\subseteq Q$ for all $e\in E(\mathcal{H})$.
\end{lemma}

\begin{proof}
The statement of the lemma holds if
$\mathcal{H}$ has at most one hyperedge. So let us assume that $\mathcal{H}$ has at least two hyperedges, say $e$ and $f$.
Let $P = e \cap f$. Since $\mathcal{H}$ is $1$-Sperner, $|P| = k-1$.
If all hyperedges of $\mathcal{H}$ contain $P$, then we are done.

If there is a hyperedge $g$ such that $P\nsubseteq g$, say $u\in P\setminus g$, then $e$ and $f$ are the only hyperedges
containing $P$, since otherwise
$g$ should contain all  vertices of such hyperedges other than $u$, which
would imply $|g|>k$.
Consequently, all hyperedges that miss a vertex of $P$ are subsets of $Q = e\cup f$, and
the lemma is proved.
\end{proof}

Lemma~\ref{lem2} suggests the following two families of uniform $1$-Sperner hypergraphs.

\begin{example}\label{ex:stars}
Given $k\ge 1$, an {\em  $k$-star} is a $k$-uniform hypergraph $\mathcal{H} = (V,E)$
such that  there exists sets $X,Y\subseteq V$ such that
\begin{itemize}
  \item $X\cup Y\subseteq V$ where $|X| = k-1$, $Y\neq\emptyset$, and $X\cap Y = \emptyset$, and
  \item $E = \{X\cup\{y\}\mid y\in Y\}$.
\end{itemize}
If this is the case, we say that $\mathcal{H}$ is the {\em ($k$-)star generated by $(V,X,Y)$}.

\begin{sloppypar}
Clearly, every $k$-star is $1$-Sperner.
Moreover, let us verify that every $k$-star is $z$-decomposable with respect to every vertex $z$.
Let $\mathcal{H}$ be a $k$-star generated by $(V,X,Y)$ and let $z\in V(\mathcal{H})$.
If $z\in X$, then $k\ge 2$ and we have  $\mathcal{H} = \mathcal{H}_1\odot \mathcal{H}_2$ where
$\mathcal{H}_1$ is the $(k-1)$-star generated by $X\setminus \{z\}$ and $Y$ and
$V(\mathcal{H}_2) = E(\mathcal{H}_2) = \emptyset$.
If $z\in Y$, then we have $\mathcal{H} = \mathcal{H}_1\odot \mathcal{H}_2$ where
$\mathcal{H}_1 = (X,\{X\})$ and
$\mathcal{H}_2 = (Y\setminus\{z\},\{\{y\}\mid y\in Y\setminus\{z\}\})$.
Finally, if $z\in V\setminus(X\cup Y)$, then $z$ is isolated and
$\mathcal{H}$ is $z$-decomposable by Proposition~\ref{prop:complement-2}.
\end{sloppypar}
\end{example}

\begin{example}\label{ex:co-stars}
Given $k\ge 1$, an {\em  $k$-antistar} is a $k$-uniform hypergraph $\mathcal{H} = (V,E)$
such that  there exists sets $X,Y\subseteq V$ such that
\begin{itemize}
  \item $X\cup Y\subseteq V$ where $Y\neq\emptyset$ and $X\cap Y = \emptyset$, and
  \item $E = \{X\cup (Y\setminus\{y\})\mid y\in Y\}$.
\end{itemize}
If this is the case, we say that $\mathcal{H}$ is the {\em ($k$-)antistar generated by $(V,X,Y)$}. Note that every $k$-antistar is the complement of a $k$-star. It follows, using
Propositions~\ref{prop:complement} and~\ref{prop:complement-2} and the properties
of stars observed in Example~\ref{ex:stars}, that
every antistar is $1$-Sperner and $z$-decomposable with respect to each vertex $z$.
\end{example}

\section{Decomposition theorem}\label{sec:decomposition}

To prove the main structural result about $1$-Sperner hypergraph (Theorem~\ref{thm:decomposition}), we need the following technical lemma.

\begin{lemma}\label{lem1}
Let $\mathcal{H}$ be a $1$-Sperner hypergraph with $E(\mathcal{H})\neq\emptyset$ and let $C$ be a hyperedge of $\mathcal{H}$ of maximum size. Then, for every two distinct vertices $x, y\not\in C$ and every two hyperedges $A$ containing $x$ and $B$ containing $y$, $|A|\le |B|$ implies $A\cap C\subseteq B\cap C$.
\end{lemma}

\begin{proof}
Note that $|A|\le |C|$, therefore
$A\setminus C = \{x\}$, since $\mathcal{H}$ is $1$-Sperner. Analogously,
$B\setminus C = \{y\}$. Thus, if the sets $A\cap C$ and $B\cap C$ were not comparable with respect to inclusion, the pair $\{A,B\}$ would violate the $1$-Sperner property of $\mathcal{H}$.
\end{proof}

\begin{sloppypar}
\begin{theorem}\label{thm:decomposition}
Every $1$-Sperner hypergraph $\mathcal{H} = (V,E)$ with $V\neq \emptyset$
is $z$-decomposable for some $z\in V(\mathcal{H})$, that is, it is the gluing of two $1$-Sperner hypergraphs.
\end{theorem}
\end{sloppypar}

\begin{proof}
By Proposition~\ref{prop:complement-2}, we may assume that $\mathcal{H}$ does not have any isolated vertices.
For every $v\in V$, let $$k(v) =
\max_{v\in e\in E}|e|$$
and let $k = \max_{e\in E}|e|$.

We consider two cases.

\smallskip

\noindent\emph{Case 1: Not all the $k(v)$ values are the same.}
Let $v\in V$ be a vertex with the smallest $k(v)$ value.
Then $k(v)<k$ by the assumption of this case.

Suppose first that for every hyperedge $f\in E$ such that
$v\not\in f$, we have $|f|\ge k(v)$. We claim that in this case $\mathcal{H}$ is $v$-decomposable.
This is because for every two hyperedges $e,f\in E$ such that $v\in e\setminus f$, we have
$|f|\ge k(v)\ge |e|$, implying $|f\setminus e|\ge |e\setminus f|$, from what we derive, using the fact that
$\mathcal{H}$ is $1$-Sperner, that $|e\setminus f| = 1$, that is, $e\setminus\{v\}\subseteq f$.
This proves the claim.

Assume next that there exists a hyperedge $f\in E$ such that $v\not\in f$ and $|f|<k(v)$.
Let $e$ be a hyperedge containing $v$ of size $k(v)$, and let $g$ be a hyperedge of maximum size,
that is, $|g| = k$. Then $v\not\in g$, since $k(v)<|g|$.
Since $\mathcal{H}$ is Sperner, there exists a vertex $u\in f\setminus g$.
Note that $|f|\le |g|$, therefore
$f\setminus g = \{u\}$, since $\mathcal{H}$ is $1$-Sperner.
Moreover, $u\neq v$ since $u\in f$ and $f$ does not contain~$v$.

We know that $k(u)\ge k(v)$, by our choice of $v$. Therefore, there exists a hyperedge $h$ containing $u$ and of size
$k(u)$. Since $|h| = k(u)\ge k(v)>|f|$, we have $h\neq f$.
Applying Lemma~\ref{lem1} with $(x,y,A,B,C) = (u,v,f,e,g)$ implies
$f\cap g\subseteq e\cap g$.
Applying Lemma~\ref{lem1} with $(x,y,A,B,C) = (v,u,e,h,g)$ implies
$e\cap g\subseteq h\cap g$.
Consequently, $f\cap g\subseteq h\cap g$.
On the other hand, $f\setminus g = h\setminus g = \{u\}$.
It follows that $f\subseteq h$, contradicting the Sperner property of $\mathcal{H}$.
This completes Case~1.

\medskip
\noindent\emph{Case 2: All the $k(v)$ values are the same.}
Let $v\in V(\mathcal{H})$ and let $k = k(v)$. If $k\le 1$, then $\mathcal{H}$ is $z$-decomposable with respect to every vertex $z$.
So suppose that $k\ge 2$. Consider the subhypergraph $\mathcal{H}'$ of $\mathcal{H}$ with $V({\mathcal{H}'}) = V(\mathcal{H})$
formed by the hyperedges of $\mathcal{H}$ of size $k$.
By Lemma~\ref{lem2} applied to $\mathcal{H}'$, either there is a subset $P$ of vertices of size $k-1$ such that $P\subseteq e$ for all
 $e\in E(\mathcal{H}')$ or there is a subset $Q$ of vertices of size $k+1$ such that $e\subseteq Q$ for all $e\in E(\mathcal{H}')$.

Suppose first that there is a subset $P$ of vertices of size $k-1$ such that $P\subseteq e$ for all $e\in E(\mathcal{H}')$.
If $\mathcal{H}' = \mathcal{H}$, that is, all hyperedges of $\mathcal{H}$ are of size $k$, then $\mathcal{H}$ is $z$-decomposable
with respect to every vertex $z$ (cf.~Example~\ref{ex:stars}).
So we may assume that $\mathcal{H}' \neq \mathcal{H}$, that is, that $\mathcal{H}$ contains a hyperedge $g$ of size less than $k$.
By the assumption of Case~$2$, we know that $g\subseteq \cup_{f\in E(\mathcal{H}')}f$.
Since $\mathcal{H}$ is Sperner, $g$ is not contained in any of the hyperedges of $\mathcal{H}'$; moreover
$g$ contains at least two vertices from the set
$Y = \left(\cup_{f\in E(\mathcal{H}')}f\right)\setminus P$.
If $g$ contains at least three vertices from $Y$, say $y_1, y_2, y_3$, then the hyperedges
$P\cup \{y_1\}$ and $g$ would violate the $1$-Sperner property, since
$\{y_2,y_3\}\subseteq g\setminus (P\cup \{y_1\})$ and
$P\setminus g\subseteq (P\cup \{y_1\})\setminus g$  (note that $|P\setminus g|\ge 3$).
It follows that $|g\cap Y| = 2$.
In fact, we have $|Y| = 2$, say $Y = \{y_1,y_2\}$, since otherwise, using similar arguments as above, we see that
the sets $P\cup \{y\}$ and $g$ would violate the $1$-Sperner property, where $y\in Y\setminus g$.
It follows that $\mathcal{H}'$ has exactly $2$ hyperedges, and $Y\subseteq e$ for every set
$e\in E(\mathcal{H})\setminus E(\mathcal{H}')$.
Consequently, $\mathcal{H}$ is $y$-decomposable for every $y\in Y$:
Decomposing $\mathcal{H}$ with respect to
$y = y_1$, for instance, we have $\mathcal{H} = \mathcal{H}_1\odot \mathcal{H}_2$ where
$\mathcal{H}_1 = (V_1,E_1)$ with $V_1 = V\setminus\{y_1\}$,
$E_1 = \{e\setminus\{y_1\}\mid y_1\in e\in E(\mathcal{H})\}$, and
$\mathcal{H}_2 = (\emptyset, \{\emptyset\})$.

It remains to consider the case when there is a subset $Q$ of vertices of size $k+1$ such that $e\subseteq Q$ for all $e\in E(\mathcal{H}')$.
Since we assume that $k(v) = k$ for all vertices $v$, we have $V = Q$.
Let us define
$X = \cap_{f\in E(\mathcal{H}')}f$ and
$Y = V\setminus X$.
Then, $Y\neq\emptyset$, and every hyperedge $g\in E(\mathcal{H})\setminus E(\mathcal{H}')$ must contain $Y$,
since $\mathcal{H}$ is Sperner and for every vertex $y\in Y$, the set $V\setminus \{y\}$ is a hyperedge of $\mathcal{H}$.
Consequently, $\mathcal{H}$ is $y$-decomposable for every $y\in Y$:
taking any $y\in Y$, we have $\mathcal{H} = \mathcal{H}_1\odot \mathcal{H}_2$ where
$\mathcal{H}_1 = (V_1,E_1)$ with $V_1 = V\setminus\{y\}$,
$E_1 = \{e\setminus\{y\}\mid y\in e\in E(\mathcal{H})\}$, and
$\mathcal{H}_2 = (\emptyset, \{\emptyset\})$.
\end{proof}

Let us say that a gluing of two vertex-disjoint $1$-Sperner hypergraphs $\mathcal{H}_1 = (V_1,E_1)$, $\mathcal{H}_2 = (V_2,E_2)$
is \emph{safe} if it results in a $1$-Sperner hypergraph. By Proposition~\ref{prop:gluing}, this is
always the case unless $E_1 = \{V_1\}$ and $E_2 = \{\emptyset\}$.
Thus, Theorem~\ref{thm:decomposition} and Proposition~\ref{prop:gluing} imply the following composition result
for the class of \hbox{$1$-Sperner} hypergraphs.

\begin{sloppypar}
\begin{theorem}\label{cor:composition}
A hypergraph $\mathcal{H}$ is $1$-Sperner if and only if it either has no vertices
(that is, $\mathcal{H}\in \{(\emptyset, \emptyset),(\emptyset, \{\emptyset\})\}$)
or it is a safe gluing of two smaller $1$-Sperner hypergraphs.
\end{theorem}
\end{sloppypar}

\section{Applications of the decomposition theorem}\label{sec:applications}

In this section we present several applications of Theorems~\ref{thm:decomposition} and~\ref{cor:composition} giving further insight on the properties of $1$-Sperner hypergraphs.

\subsection{$1$-Sperner hypergraphs are threshold}
\label{sec:1-Sperner-threshold}

Our first application is motivated by the result of~\citet{MR3281177} stating that every dually Sperner hypergraph is threshold, see Theorem~\ref{thm:dually-Sperner}. The proof of Theorem~\ref{thm:dually-Sperner} given in~\cite{MR3281177}
is based on the characterization of thresholdness in terms of asummability (see Theorem~\ref{thm:ChowElgot}) and does not show how to compute a threshold separator of a dually Sperner hypergraph.
Here we give an alternative proof of Theorem~\ref{thm:dually-Sperner}, based on the composition theorem of $1$-Sperner hypergraphs. In contrast with the proof from~\cite{MR3281177}, this proof is constructive in the sense that it computes an explicit threshold separator of a $1$-Sperner or, more generally, dually Sperner hypergraph.

Clearly, every $1$-Sperner hypergraph is dually Sperner. Therefore, Theorem~\ref{thm:dually-Sperner} implies that every
$1$-Sperner hypergraph is threshold. We will now derive this fact
directly from the composition theorem. In fact, we show a bit more, namely that every $1$-Sperner hypergraph $\mathcal{H} = (V,E)$ admits a \emph{positive} threshold separator, that is, a threshold separator $(w,t)$ such that $w:V\to \mathbb{Z}_{>0}$ is a (strictly) positive integer weight function.

The following simple technical claim will be used in the proof.

\begin{lemma}\label{lem:technical}
For every threshold separator $(w,t)$ of a Sperner threshold hypergraph $\mathcal{H} = (V,E)$, we have:
\begin{enumerate}[(i)]
  \item if $w(V)= t$ then $E = \{V\}$, and
  \item if $t=0$ then $E = \{\emptyset\}$.
\end{enumerate}
\end{lemma}

\begin{proof}
If $w(V)= t$, then $V$ is a hyperedge and if
$t=0$, then the empty set is a hyperedge. (Both of these claims follow from the fact that $(w,t)$ is a threshold separator of $\mathcal{H}$.) In both cases no other hyperedge may exist due to the Sperner property.
\end{proof}

\begin{theorem}\label{thm:1-Sperner-threshold}
Every $1$-Sperner hypergraph is threshold with a positive threshold separator.
\end{theorem}

\begin{proof}
Let $\mathcal{H} = (V,E)$ be a $1$-Sperner hypergraph.
The proof is by induction on $n = |V|$.
For $n = 0$, we can obtain a positive threshold separator by taking the (empty) mapping given by $w(x) = 1$ for all $x\in V$
and the threshold $$t = \left\{
                         \begin{array}{ll}
                           1, & \hbox{if $E = \emptyset$;} \\
                           0, & \hbox{if $E = \{\emptyset\}$.} \\
                         \end{array}
                       \right.
$$

Now, let $n\ge 1$. By Theorem~\ref{cor:composition},  $\mathcal{H}$ is the safe gluing of two $1$-Sperner hypergraphs, say
$\mathcal{H} = \mathcal{H}_1\odot  \mathcal{H}_2$ with
$\mathcal{H}_1 = (V_1,E_1)$ and
$\mathcal{H}_2 = (V_2,E_2)$, where
$V = V_1\cup V_2\cup\{z\}$, $V_1\cap V_2 = \emptyset$, and $z\not\in V_1\cup V_2$.
By the inductive hypothesis, $\mathcal{H}_1$ and $\mathcal{H}_2$ admit positive threshold separators.
That is, there exist positive integer weight functions $w_i:V_i\to \mathbb{Z}_{>0}$ and
non-negative integer thresholds $t_i\in \mathbb{Z}_{\ge 0}$ for $i = 1,2$ such that
for every subset $X\subseteq V_i$, we have
$w_i(X)\ge t_i$ if and only if $e\subseteq X$ for some $e\in E_i$.

Let us define the threshold $t = Mw_1(V_1)+t_2$, where $M = w_2(V_2)+1$, and
the weight function $w:V\to \mathbb{Z}_{>0}$ by the rule
$$w(x) = \left\{
          \begin{array}{ll}
            Mw_1(x), & \hbox{if $x\in V_1$;} \\
            w_2(x), & \hbox{if $x\in V_2$;} \\
            M(w_1(V_1)-t_1)+t_2, & \hbox{if $x = z$.}
          \end{array}
        \right.$$
We claim that $(w,t)$ is a positive threshold separator of $\mathcal{H}$.
Let us first verify that the so defined weight function is indeed positive.
Since $w_i$ for $i \in \{1,2\}$ are positive and $M>0$, we have $w(x)>0$ for all $x\in V_1\cup V_2$.
Moreover, since $w_1(V_1)\ge t_1$, $M\ge 0$, and $t_2\ge 0$, we have
$w(z)\ge 0$. If $w(z) = 0$, then $w_1(V_1)=t_1$ and $t_2 = 0$,
which by \Cref{lem:technical} implies $E_1 = \{V_1\}$ and $E_2 = \{\emptyset\}$, contrary to the fact that the gluing is safe. It follows that $w(z)>0$, as claimed.

Next, we verify that $(w,t)$ is a threshold separator of $\mathcal{H}$, that is, that for every subset $X\subseteq V$, we have
$w(X)\ge t$ if and only if $e\subseteq X$ for some $e\in E$.

Suppose first that $w(X)\ge t$ for some $X\subseteq V$.
Let $X_i = X\cap V_i$ for $i = 1,2$.
For later use, we note that
\begin{equation}\label{ineq4}
w_2(X_2)\le w_2(V_2)<M\,.
\end{equation}
Suppose first that $z\in X$. Then
\begin{eqnarray*}
Mw_1(V_1)+t_2 &=& t ~\le~ w(X)\\
&=&  w(z)+w(X_1)+w(X_2)\\
&=& M(w_1(V_1)-t_1)+t_2 + Mw_1(X_1)+w_2(X_2)\,,
\end{eqnarray*}
which implies
\begin{equation}\label{ineq3}
Mw_1(X_1)+w_2(X_2)\ge Mt_1\,.
\end{equation}
If $w_1(X_1) \le t_1-1$ then, using~\eqref{ineq4}, we obtain
$$Mw_1(X_1)+w_2(X_2)\le
Mt_1-M+w_2(X_2)<Mt_1\,,$$
a contradiction with~\eqref{ineq3}.
It follows that $w_1(X_1) \ge t_1$.
Consequently there exists $e_1\in E_1$ such that $e_1\subseteq X_1$,
hence the hyperedge $e:=\{z\}\cup e_1\in E$ satisfies
$e\subseteq X$.

Now, suppose that $z\not\in X$. In this case,
$$Mw_1(V_1)+t_2 = t \le w(X) = w(X_1)+w(X_2) = Mw_1(X_1)+w_2(X_2)\,,$$
which implies
\begin{equation}\label{ineq1}
Mw_1(X_1)+w_2(X_2)\ge Mw_1(V_1)+t_2\,.
\end{equation}
We must have $X_1 = V_1$ since if there exists a vertex $v\in V_1\setminus X_1$, then
we would have
\begin{eqnarray*}
Mw_1(X_1)+w_2(X_2)&\le& Mw_1(V_1)-Mw_1(v)+w_2(X_2)\\
&\le& Mw_1(V_1)-M+w_2(X_2)\\
&<&Mw_1(V_1)+t_2\,,
\end{eqnarray*}
where the last inequality follows from~\eqref{ineq4} and $t_2\ge 0$.
Therefore, inequality~\eqref{ineq1} simplifies to
$w_2(X_2)\ge t_2$, and consequently there exists a hyperedge $e_2\in E_2$ such that
$e_2\subseteq X_2$. This implies that $\mathcal{H}$ has a hyperedge $e := V_1\cup e_2$ such that
$e\subseteq  V_1\cup X_2 = X$.

For the converse direction, suppose that $X$ is a subset of $V$ such that
$e\subseteq X$ for some $e\in E$.  We need to show that $w(X)\ge t$.
We consider two cases depending on whether $z\in e$ or not.
Suppose first that $z\in e$.
Then $e = \{z\}\cup e_1$ for some $e_1\in E_1$.
Due to the property of $w_1$, we have $w_1(e_1)\ge t_1$.
Consequently,
\begin{eqnarray*}
w(X) &\ge& w(e) \\
&=& w(z)+w(e_1)\\
&=& M(w_1(V_1)-t_1)+t_2 + Mw_1(e_1)\\
&\ge& Mw_1(V_1)-Mt_1+t_2 + Mt_1\\
&= & Mw_1(V_1)+t_2~=~t\,.
\end{eqnarray*}
Suppose now that $z\not\in e$.
Then $e = V_1\cup e_2$ for some $e_2\in E_2$.
Due to the property of $w_2$, we have $w_2(e_2)\ge t_2$.
Consequently,
\begin{eqnarray*}
w(X) &=& w(V_1)+w(e_2) \\
&=& Mw_1(V_1)+w_2(e_2)\\
&\ge & Mw_1(V_1)+t_2~=~t\,.
\end{eqnarray*}
This shows that $w(X)\ge t$ whenever $X$ contains a hyperedge of $\mathcal{H}$, and completes the proof.
\end{proof}

We now give an alternative proof of Theorem~\ref{thm:dually-Sperner} announced above.

\noindent{\bf An alternative proof of Theorem~\ref{thm:dually-Sperner}.}
Let $\mathcal{H}=(V,E)$ be a dually Sperner hypergraph. By Observation~\ref{obs:Sp-red}, its Sperner reduction is $1$-Sperner. By Theorem~\ref{thm:1-Sperner-threshold}, $\Sp(\mathcal{H})$ has a positive threshold separator, say $(w,t)$. Since $(w,t)$ is a threshold separator of $\Sp(\mathcal{H})$, it is also a threshold separator of $\mathcal{H}$, by \Cref{lem:Sperner-reduction}. Thus, $\mathcal{H}$ is threshold.\qed

\medskip
We would like to emphasize that the above proof implies the following simple efficient procedure of obtaining a threshold separator of a given dually Sperner hypergraph $\mathcal{H}$:
\begin{enumerate}[(1)]
  \item compute its Sperner reduction, $\Sp(\mathcal{H})$, and
  \item construct a positive threshold separator $(w,t)$ of $\Sp(\mathcal{H})$ recursively along a decomposition of $\Sp(\mathcal{H})$ into smaller $1$-Sperner hypergraphs given by Theorem~\ref{cor:composition} (eventually resulting in trivial $1$-Sperner hypergraphs).
\end{enumerate}
Then $(w,t)$ is a threshold separator of $\mathcal{H}$.

\subsection{Further relations between threshold, equilizable, and
$1$-Sperner hypergraphs}

The same inductive construction of a threshold separator as that given in the proof of Theorem~\ref{thm:1-Sperner-threshold} shows that every $1$-Sperner hypergraph is also equilizable (see Section~\ref{sec:background} for the definition).

\begin{theorem}\label{thm:equilizable}
Every $1$-Sperner hypergraph is equilizable.
\end{theorem}

Theorem~\ref{thm:equilizable} can be proved by slightly modifying the
above proof of Theorem~\ref{thm:1-Sperner-threshold}; for the sake of completeness, we include it in Appendix. Theorem~\ref{thm:equilizable} will be used in Section~\ref{sec:bounds} to establish an upper bound on the size of a $1$-Sperner hypergraph of a given order.

\begin{sloppypar}
Combining Theorems~\ref{thm:1-Sperner-threshold} and~\ref{thm:equilizable} shows that every $1$-Sperner hypergraph is threshold and equilizable.
In particular, the properties of thresholdness and equilizability trivially coincide within the class of $1$-Sperner hypergraphs. This raises the question of whether the two properties are comparable in the larger class of Sperner hypergraphs. This is not the case. As the following two examples show, the properties of thresholdness and equilizability are incomparable in the class of Sperner hypergraphs.
\end{sloppypar}

\medskip

\begin{example}\label{example-equilizable-not-threshold}
The following Sperner hypergraph is equilizable but not threshold:
$\mathcal{H}_1 = (V_1,E_1)$
where $V_1 = \{v_1,v_2,v_3,v_4,v_5\}$, $E_1 = \{\{v_1,v_2\}, \{v_2,v_3,v_4\}, \{v_4,v_5\}\}$. The function $w:V_1\to \mathbb{Z}_{\ge 0}$ defined by
$w(v_1) = 5$, $w(v_2) = 4$, $w(v_3) = 3$, $w(v_4) = 2$, and $w(v_5) = 7$ assigns a total weight of~$9$ to each hyperedge and to no other subset of $V_1$. Thus,~$\mathcal{H}_1$ is equilizable.
To see that $\mathcal{H}_1$ is not threshold, note that any threshold separator $(w',t')$ of $\mathcal{H}_1$ would have to satisfy $w'(v_1)+w'(v_2)\ge t'$ and $w'(v_3)+w'(v_4)\ge t'$, as well as $w'(v_1)+w'(v_3)<t'$ and $w'(v_2)+w'(v_4)<t'$, which is impossible. In other words, $\mathcal{H}_1$ fails to be threshold since it is not $2$-asummable;
cf.~Theorem~\ref{thm:ChowElgot}.
\end{example}

\medskip
\begin{example}\label{example-threshold-not-equilizable}
The following Sperner hypergraph is threshold but not equilizable:
$\mathcal{H}_2 = (V_2,E_2)$
where $V_2 = \{v_1,v_2,v_3,v_4\}$, $E_2 = \{\{v_1,v_2\},
\{v_1,v_3\}, \{v_2,v_3\}, \{v_2,v_4\}, \{v_3,v_4\}\}$.
The function $w:V_2\to \mathbb{Z}_{\ge 0}$ defined by
$w(v_1) = w(v_4) = 1$, $w(v_2) = w(v_3) = 2$, and threshold $t = 3$
form a threshold separator of $\mathcal{H}_2$.
Thus, $\mathcal{H}_2$ is threshold.
To see that $\mathcal{H}_2$ is not equilizable, note that any function $w':V_2\to \mathbb{Z}_{\ge 0}$ such that the total weight of every hyperedge is the same, say $t'$, must assign weight $t'/2$ to every vertex. Consequently, the set $\{v_1,v_4\}$, which is not a hyperedge, would also be of total weight $t'$.
\end{example}

\medskip
Furthermore, the following examples show that there exist Sperner hypergraphs that are threshold and equilizable but not $1$-Sperner.

\begin{example}\label{example-threshold-equilizable-not-1-Sperner}
For every $k\ge 2$ and $n\ge 2k$, the complete $k$-uniform hypergraph $\mathcal{H}_{n,k}$ defined with $V(\mathcal{H}_{n,k}) = \{1,\ldots, n\}$ and $E(\mathcal{H}_{n,k}) = \{X\mid X\subseteq \{1,\ldots, n\}, |X| = k\}$ is not $1$-Sperner, but it is both threshold and equilizable, as verified by the weight function that is constantly equal $1$ and threshold $t = k$.
\end{example}

\subsection{Bounds on the size of $1$-Sperner hypergraphs}
\label{sec:bounds}

We now establish some upper and lower bounds on the number of hyperedges in a $1$-Sperner hypergraph with a given number of vertices. By $\bbzero$, resp.~$\bbone$, we will denote the vector of all zeroes, resp.~ones, of appropriate dimension (which will be clear from the context).
The following lemma can be easily derived from Theorem~\ref{thm:equilizable}.

\begin{lemma}\label{lem:x}
For every $1$-Sperner hypergraph $\mathcal{H} = (V,E)$ such that
\hbox{$E\neq \emptyset$} and
\hbox{$E\neq \{\emptyset\}$},
there exists a vector
$x\in \mathbb{R}^V_{\ge 0}$ such that $A^\mathcal{H}x = \bbone$ and $\bbone^\top x \ge 1$.
\end{lemma}

\begin{sloppypar}
\begin{proof}
Let $\mathcal{H} = (V,E)$ be a $1$-Sperner hypergraph as in the statement of the lemma.
By Theorem~\ref{thm:equilizable}, $\mathcal{H}$ is equilizable. Let
$w:V\to \mathbb{Z}_{\ge 0}$ be a non-negative integer weight function and
$t\in \mathbb{Z}_{\ge 0}$ a non-negative integer threshold
such that for every subset $X\subseteq V$, we have
$w(X) = t$ if and only if $X\in E$.
If $t = 0$, then $\emptyset\in E$ and consequently $E = \{\emptyset\}$, a contradiction.
It follows that $t>0$, and we can define
the vector $x\in \mathbb{R}^V_{\ge 0}$ given by
$x_v = w(v)/t$ for all $v\in V$. We claim that vector $x$ satisfies the desired properties
$A^\mathcal{H}x = \bbone$ and $\bbone^\top x \ge 1$.

Since $w(X) = t$ for all $X\in E$, we have
$A^\mathcal{H}x = \bbone$.
Since $E\neq \emptyset$, an arbitrary hyperedge $e\in E$
shows that $t = w(e) \le w(V)$. Consequently, we also have
\hbox{$\bbone^\top x = \sum_{v\in V}x_v = w(V)/t\ge 1\,.$}
\end{proof}
\end{sloppypar}

\begin{corollary}\label{cor:lambda}
For every $1$-Sperner hypergraph $\mathcal{H} = (V,E)$
and every vector $\lambda\in \mathbb{R}^E$
we have
$$\lambda^\top A^\mathcal{H} = \bbone^\top \Rightarrow \lambda^\top\bbone \ge 1\,.$$
\end{corollary}

\begin{proof}
If \hbox{$E= \emptyset$} or
\hbox{$E= \{\emptyset\}$}, then the left hand side of the above implication is always false.
In all other cases, by Lemma~\ref{lem:x}, there exists a vector $x\in \mathbb{R}^V$ such that $A^\mathcal{H}x = \bbone$ and $\bbone^\top x \ge 1$.
Therefore, equation
$\lambda^\top A^\mathcal{H} = \bbone^\top$ implies
$\lambda^\top\bbone = \lambda^\top A^\mathcal{H}x = \bbone^\top x \ge 1$.
\end{proof}

The composition theorem and the above corollary imply the following useful property of $1$-Sperner hypergraphs.

\begin{theorem}\label{prop:linear-independence}
For every $1$-Sperner hypergraph $\mathcal{H} = (V,E)$ such that \hbox{$E\neq \{\emptyset\}$},
the characteristic vectors of its hyperedges are linearly independent
(over the field of real numbers).
\end{theorem}

\begin{proof}
We use induction on $|V|$. If $|V|\le 1$, then the statement holds since $E\neq \{\emptyset\}$.

Suppose now that $|V|>1$. Then by  Theorem~\ref{thm:decomposition},
$\mathcal{H}$ is the gluing of two $1$-Sperner hypergraphs, say
$\mathcal{H} = \mathcal{H}_1\odot  \mathcal{H}_2$ with
$\mathcal{H}_1 = (V_1,E_1)$ and
$\mathcal{H}_2 = (V_2,E_2)$, where
$V = V_1\cup V_2\cup\{z\}$, $V_1\cap V_2 = \emptyset$, and $z\not\in V_1\cup V_2$.

Let $\lambda\in \mathbb{R}^E$ be a vector such that
$\lambda^\top A^\mathcal{H} = \bbzero$.
Let $\lambda^1$ and $\lambda^2$ be the restrictions of $\lambda$ to the
hyperedges corresponding to $E_1$ and
$E_2$, respectively.
The equation
$\lambda^\top A^\mathcal{H} = \bbzero$
implies the system of equations
\begin{eqnarray*}
  (\lambda^1)^\top\bbone &=& 0\in \mathbb{R}\,, \\
  (\lambda^1)^\top A^{\mathcal{H}_1}+((\lambda^2)^\top \bbone)\bbone^\top  &=& \bbzero^\top\in \mathbb{R}^{V_1}\,, \\
  (\lambda^2)^\top  A^{\mathcal{H}_2} &=& \bbzero^\top\in \mathbb{R}^{V_2}\,.
\end{eqnarray*}

In all cases, the inductive hypothesis implies that
$\lambda^1 = \bbzero^\top\in \mathbb{R}^{E_1}$ and
$\lambda^2 = \bbzero^\top\in \mathbb{R}^{E_2}$, except in the case
when $E_2 = \{\emptyset\}$. In this case, $\lambda_2$ is a single number,
say
$\lambda^*$.
If $\lambda^* = 0$, then $\lambda_1 = \bbzero^\top$ follows by the induction hypothesis.
If $\lambda^* \neq 0$, then
$\hat \lambda := -\lambda_1/\lambda^*$ satisfies
$\hat \lambda^\top A^{\mathcal{H}_1}= \bbone^\top$ and
$\hat \lambda^\top \bbone = \bbzero$,
contradicting Corollary~\ref{cor:lambda}.
\end{proof}

Since the characteristic vectors of the hyperedges of an $n$-vertex $1$-Sperner hypergraph are linearly independent vectors in $\mathbb{R}^n$,
we obtain the following upper bound on the size of a $1$-Sperner hypergraph
in terms of its order.

\begin{corollary}\label{cor:edges}
For every $1$-Sperner hypergraph $\mathcal{H} = (V,E)$ with $V\neq \emptyset$, we have $|E|\le |V|$.
\end{corollary}

\begin{sloppypar}
The bound $|E| \le |V|$ can also be proved more directly from the decomposition theorem (Theorem~\ref{cor:composition}), using induction on the number of vertices and analyzing various cases according to whether the two constituent hypergraphs have non-empty vertex set or not.
We decided to include the proof based on~\Cref{prop:linear-independence}, since linear independence is an interesting property of $1$-Sperner hypergraphs and the inequality $|E|\le |V|$ is just one consequence of that.
\end{sloppypar}

\medskip
We now turn to the lower bound. Recall that a vertex $u$ in a hypergraph $\mathcal{H} = (V,E)$ is said to be \emph{universal} (resp., \emph{isolated}) if it is contained in all (resp., in no) hyperedges. Moreover, two vertices $u,v$ of a hypergraph $\mathcal{H} = (V,E)$ are \emph{twins} if
they are contained in exactly the same hyperedges.

Corollary~\ref{cor:edges} gives an upper bound on the size of a $1$-Sperner hypergraph in terms of its order. Can we prove a lower bound of a similar form? In general not, since adding universal vertices, isolated vertices, or twin vertices preserves the $1$-Sperner property and the size, while it increases the order. However, as we show next, for $1$-Sperner hypergraphs without universal, isolated, and twin vertices, the following sharp lower bound on the size in terms of the order holds.

\begin{proposition}\label{prop:lower-bound}
For every $1$-Sperner hypergraph $\mathcal{H} = (V,E)$
with $|V|\ge 2$ and without universal, isolated, and twin vertices,
we have the following sharp lower bound
$$|E|\ge \left\lceil\frac{|V|}{2}\right\rceil+1\,.$$
\end{proposition}

\begin{proof}
We use induction on $n = |V|$. For $n\in \{2,3,4\}$, it can be easily verified that the statement holds.

Now, let $\mathcal{H} = (V,E)$ be a $1$-Sperner hypergraph with $n\ge 5$ and without universal vertices, isolated vertices, and twin vertices.
By Theorem~\ref{thm:decomposition},
$\mathcal{H}$ is the gluing of two \hbox{$1$-Sperner} hypergraphs, say
$\mathcal{H} = \mathcal{H}_1\odot  \mathcal{H}_2$ with
$\mathcal{H}_1 = (V_1,E_1)$ and
$\mathcal{H}_2 = (V_2,E_2)$, where
$V = V_1\cup V_2\cup\{z\}$, $V_1\cap V_2 = \emptyset$, and $z\not\in V_1\cup V_2$.
Since $\mathcal{H}$ has no twins, $\mathcal{H}_1$ and $\mathcal{H}_2$ also have no twins.
Let $n_i = |V_i|$ and $m_i = |E_i|$ for $i = 1,2$, and let $m = |E|$.

We have $m = m_1+m_2$, and by the rules of the gluing, $n = n_1+n_2+1$.
By Proposition~\ref{prop:complement}, we may assume that $n_1\ge n_2$ (otherwise, we can consider the complementary hypergraph).
In particular, $n_1\ge 3$.
The fact that $\mathcal{H}$ does not have a universal vertex implies
$\mathcal{H}_1$ does not have a universal vertex. Similarly,
$\mathcal{H}_2$ does not have an isolated vertex.
Since $\mathcal{H}$ does not have any pairs of twin vertices, we have that either $\mathcal{H}_1$ does not have a isolated vertex, or
$\mathcal{H}_2$ does not have a universal vertex. We may assume that
$\mathcal{H}_2$ does not have a universal vertex (otherwise, we consider a different gluing in which we delete the universal vertex from $\mathcal{H}_2$ and add an isolated vertex to $\mathcal{H}_1$).
Since $\mathcal{H}_2$ is a Sperner hypergraph without an isolated or a universal vertex, we have $n_2\neq 1$.

Suppose first that $n_2\ge 2$. We apply the inductive hypothesis for $\mathcal{H}_1'$ and $\mathcal{H}_2$, where
$\mathcal{H}_1'$ is the hypergraph obtained from $\mathcal{H}_1$ by deleting from it the isolated vertex (if it exists).
Letting $n_1' = |V(\mathcal{H}_1')|$ and
$m_1' = |E(\mathcal{H}_1')|$, we thus have $n_1'\ge n_1-1$ and also $n_1'\ge 2$.
We obtain
$$m_1 = m_1'\ge \frac{n_1'+2}{2}\ge \frac{n_1+1}{2}$$
and
$$m_2\ge \frac{n_2}{2}+1\,.$$
Consequently,
$$m = m_1+m_2\ge \frac{n_1+1}{2}+\frac{n_2+2}{2} = \frac{n_1+n_2+3}{2} = \frac{n}{2}+1\,,$$
and, since $m$ is integer, the desired inequality $$m\ge \left\lceil\frac{n}{2}\right\rceil+1$$ follows.

Suppose now that $n_2 = 0$. In this case, since $\mathcal{H}$ does not have a universal vertex, we must have
$E_2 = \{\emptyset\}$ and $m_2 = 1$. As above, let $\mathcal{H}_1'$ be the hypergraph obtained from $\mathcal{H}_1$ by deleting from it the isolated vertex (if it exists). Letting $n_1' = |V(\mathcal{H}_1')|$ and $m_1' = |E(\mathcal{H}_1')|$, we obtain, by applying the inductive hypothesis to $\mathcal{H}_1'$,
$$m_1 = m_1'\ge \frac{n_1'}{2}+1\ge \frac{n_1+1}{2}\,,$$
which implies $$m = m_1+1\ge \frac{n_1+1}{2}+1 = \frac{n}{2}+1 = \frac{n}{2}+1\,.$$
This completes the proof of the inequality.

To see that the inequality is sharp, consider the following recursively defined family of hypergraphs $\mathcal{H}_k$ for $k\ge 2$:
\begin{itemize}
  \item $\mathcal{H}_2 = (\{v_1,v_2\},\{\{v_1\},\{v_2\}\})$.
  \item For $k>2$, we set $\mathcal{H}_k = \mathcal{H}_{k-1}'\odot  \mathcal{H}_{k-1}$  where
$\mathcal{H}_{k-1}'$ is the hypergraph obtained from a disjoint copy of $\mathcal{H}_{k-1}$ by adding to it an isolated vertex.
\end{itemize}
An inductive argument shows that for every $k\ge 2$, we have
$n_k = |V(\mathcal{H}_k)|  = 2^k-2$,
$m_k = |E(\mathcal{H}_k)|  = 2^{k-1}$, and consequently
$m_k = \left\lceil\frac{n_k}{2}\right\rceil+1\,.$
\end{proof}

\subsection{Minimal transversals of $1$-Sperner hypergraphs}

Recall that a \emph{transversal} of $\mathcal{H}$ is a set of vertices intersecting all hyperedges of $\mathcal{H}$.

\begin{theorem}\label{thm:min-transversal-of-1-Sperner}
The number of minimal transversals of every $1$-Sperner hypergraph $\mathcal{H} = (V,E)$ is at most
$$\max\left\{1,|V|,{|V|\choose 2}\right\}\,.$$
This bound is sharp.
Moreover, the family of minimal transversals of a given $1$-Sperner hypergraph $\mathcal{H} = (V,E)$ can be generated in time $\mathcal{O}(|V|^3|E|)$.
\end{theorem}

\begin{proof}
We first prove the upper bound on the size of the transversal hypergraph $\mathcal{H}^T$. We use induction on $n = |V|$. The claim is clear for $n = 0$.
For $n\in \{1,2,3\}$, the claim is that every $1$-Sperner hypergraph
of order $n$ has at most $n$ minimal transversals. This is true since for these small values of $n$, no family of pairwise incomparable subsets of an $n$-element set can have more than $n$ elements.

Now, let $\mathcal{H} = (V,E)$ be a $1$-Sperner hypergraph with $n\ge 4$.
By Theorem~\ref{thm:decomposition}, $\mathcal{H}$ is the gluing of two \hbox{$1$-Sperner} hypergraphs, say $\mathcal{H} = \mathcal{H}_1\odot  \mathcal{H}_2$ with $\mathcal{H}_1 = (V_1,E_1)$ and
$\mathcal{H}_2 = (V_2,E_2)$, where
$V = V_1\cup V_2\cup\{z\}$, $V_1\cap V_2 = \emptyset$, and $z\not\in V_1\cup V_2$. Denoting $n_i = |V_i|$ for $i\in \{1,2\}$, the inductive hypothesis implies that $|E({\mathcal H}_i^T)|\le \max\{1,n_i,{n_i\choose 2}\}$ for $i\in \{1,2\}$.
Suppose first that $n_2 = 0$.
In this case,
$E({\mathcal H}_2^T)$ is either $\emptyset$ (if
$E({\mathcal H}_2) = \{\emptyset\}$) or
$\{\emptyset\}$ (if
$E({\mathcal H}_2) = \emptyset$).
By \Cref{prop:gluing-transversal}, it suffices to show the inequality
$$\max\left\{1,n_1,{n_1\choose 2}\right\}+\max\left\{1,n_1\right\}\le \max\left\{1,n,{n\choose 2}\right\}\,.$$
Using $n\ge 4$ and consequently $n_1\ge 3$, the inequality reduces to
${n_1\choose 2}+n_1\le {n_1+1\choose 2}$, which is satisfied with equality.

Suppose now that $n_2 \ge 1$.
By \Cref{prop:gluing-transversal}, it suffices to show the inequality
$$\max\left\{1,n_1,{n_1\choose 2}\right\}+\max\left\{1,n_2,{n_2\choose 2}\right\}+n_1\le \max\left\{1,n,{n\choose 2}\right\}\,.$$ The inequality holds for any $n\ge 4$ and any $n_1$, $n_2$ such that $n_2 \ge 1$ and $n_1+n_2+1 = n$.
The details are left to the reader.

To see that the inequality is sharp, let $n\ge 3$ and consider the family
$\mathcal{H}_{n,n-1}$ of complete \hbox{$(n-1)$-uniform} hypergraphs; see Example~\ref{example-threshold-equilizable-not-1-Sperner}.
It is clear that the hypergraph $\mathcal{H}_{n,n-1}$ is $1$-Sperner.
Its transversal hypergraph is the complete $2$-uniform hypergraph
$\mathcal{H}_{n,2}$, which is of size ${n\choose 2}$.

It remains to show that the transversal hypergraph $\mathcal{H}^T$
of a given $1$-Sperner hypergraph $\mathcal{H} = (V,E)$
with $n = |V|$ and $m = |E|$ can be generated in time $\mathcal{O}(n^3m)$.
We may assume that $n\ge 1$.
The algorithm is as follows:
\begin{enumerate}
  \item First, we compute the $z$-decomposition $\mathcal{H} = \mathcal{H}_1\odot  \mathcal{H}_2$ (for some $z\in V$).
  \item Secondly, for $i\in \{1,2\}$, we recursively generate the transversal hypergraphs $\mathcal{H}_1^T$ and $\mathcal{H}_2^T$.
\item Finally, we compute the transversal hypergraph $\mathcal{H}^T$ using
\Cref{prop:gluing-transversal}.
\end{enumerate}
Let $T(n,m)$ denote the running time of this algorithm.
Step 1 of the algorithm can be done in time $\mathcal{O}(n^2m)$ by~\Cref{prop:complexity}.
For Step 2, let $\mathcal{H}_1 = (V_1,E_1)$ and $\mathcal{H}_2 = (V_2,E_2)$, where $V = V_1\cup V_2\cup\{z\}$, $V_1\cap V_2 = \emptyset$, and $z\not\in V_1\cup V_2$. Denoting $n_i = |V_i|$ and $m_i = |E_i|$ and
for $i\in \{1,2\}$, we can do Step 2 in time
$T(n_1,m_1)+T(n_2,m_2)$.
For Step 3, let $F$ be defined as in the proof of
\Cref{prop:gluing-transversal}, and notice that $F$ is not equal to its Sperner reduction if and only if one of the following happens: (i) the empty set is a minimal transversal of $\mathcal{H}_1$, (ii) the empty set is a minimal transversal of $\mathcal{H}_2$, or (iii) $\{u\}$ is a minimal transversal of $\mathcal{H}_1$ for some $u\in V_1$.
These conditions can be verified either in constant time (in cases (i) and (ii)) or in $\mathcal{O}(|E(\mathcal{H}_1^T)|) = \mathcal{O}(n_1^2)$ time (in case (iii)).

The above reasoning leads to the inequality
$T(n,m) \le \mathcal{O}(n^2m)+T(n_1,m_1)+T(n_2,m_2)$.
Since $n = n_1+n_2+1$ and $m = m_1+m_2$,  this inequality implies that
$T(n,m) = \mathcal{O}(n^3m)$, as claimed.
\end{proof}

\subsection*{Acknowledgements}

\begin{sloppypar}
The authors are grateful to Nina Chiarelli and Sylwia Cichacz for helpful discussions. The second author was partially funded by the Russian Academic Excellence Project `5-100'. The work for this paper was done in the framework of bilateral projects between Slovenia and the USA, partially financed by the Slovenian Research Agency (BI-US/$14$--$15$--$050$, BI-US/$16$--$17$--$030$, and BI-US/$18$--$19$--$029$). The work of the third author is supported in part by the Slovenian Research Agency (I$0$-$0035$, research program P$1$-$0285$, research projects N$1$-$0032$, J$1$-$6720$, and J$1$-$7051$).
\end{sloppypar}

\def\ocirc#1{\ifmmode\setbox0=\hbox{$#1$}\dimen0=\ht0 \advance\dimen0
  by1pt\rlap{\hbox to\wd0{\hss\raise\dimen0
  \hbox{\hskip.2em$\scriptscriptstyle\circ$}\hss}}#1\else {\accent"17 #1}\fi}

\section*{Appendix: Proof of Theorem~\ref{thm:equilizable}}

\begin{propositionEquilizable}[restated]
Every $1$-Sperner hypergraph is equilizable.
\end{propositionEquilizable}

\begin{proof}
We will show by induction on $n = |V|$ that
for every $1$-Sperner hypergraph $\mathcal{H} = (V,E)$
there exists a positive integer weight function $w:V\to \mathbb{Z}_{>0}$ and
a non-negative integer threshold $t\in \mathbb{Z}_{\ge 0}$  such that
for every subset $X\subseteq V$, we have
$w(X)= t$ if and only if $X\in E$.
This will establish the equilizability of $\mathcal{H}$.

{For $n = 0$, we can take the (empty) mapping given by $w(x) = 1$ for all $x\in V$
and the threshold $$t = \left\{
                         \begin{array}{ll}
                           1, & \hbox{if $E = \emptyset$;} \\
                           0, & \hbox{if $E = \{\emptyset\}$.}
                         \end{array}
                       \right.
$$}

Now, let $\mathcal{H} = (V,E)$ be a $1$-Sperner hypergraph with $n\ge 1$.
By Theorem~\ref{cor:composition}, $\mathcal{H}$ is a safe gluing of two $1$-Sperner hypergraphs, say
$\mathcal{H} = \mathcal{H}_1\odot  \mathcal{H}_2$ with
$\mathcal{H}_1
= (V_1,E_1)$ and
$\mathcal{H}_2 = (V_2,E_2)$, where
$V = V_1\cup V_2\cup\{z\}$, $V_1\cap V_2 = \emptyset$, and $z\not\in V_1\cup V_2$.
By the inductive hypothesis, $\mathcal{H}_1$ and $\mathcal{H}_2$ are equilizable, that is, there exist
positive integer weight functions $w_i:V_i\to \mathbb{Z}_{>0}$ and
non-negative integer thresholds $t_i\in \mathbb{Z}_{\ge 0}$ for $i = 1,2$ such that
for every subset $X\subseteq V_i$, we have
$w_i(X)= t_i$ if and only if $X\in E_i$.

Let us define the threshold $t = Mw_1(V_1)+t_2$, where $M = w_2(V_2)+1$, and
the weight function $w:V\to \mathbb{Z}_{>0}$ by the rule
$$w(x) = \left\{
          \begin{array}{ll}
            Mw_1(x), & \hbox{if $x\in V_1$;} \\
            w_2(x), & \hbox{if $x\in V_2$;} \\
            M(w_1(V_1)-t_1)+t_2, & \hbox{if $x = z$.}
          \end{array}
        \right.$$
Since the weight function defined above coincides with the one in the proof of Theorem~\ref{thm:1-Sperner-threshold}, this function is indeed strictly positive.

We claim that for every subset $X\subseteq V$, we have
$w(X)= t$ if and only if $X\in E$. This will establish the equilizability of $\mathcal{H}$.

Suppose first that $w(X)= t$ for some $X\subseteq V$. Let $X_i = X\cap V_i$ for $i = 1,2$.
For later use, we note that
\begin{equation}\label{ineq5}
w_2(X_2)\le w_2(V_2)<M\,.
\end{equation}
We consider two cases depending on whether $z\in X$ or not.
Suppose first that $z\in X$. Then
\begin{eqnarray*}
Mw_1(V_1)+t_2 &=& t ~=~ w(X)\\
&=&  w(z)+w(X_1)+w(X_2)\\
&=& M(w_1(V_1)-t_1)+t_2 + Mw_1(X_1)+w_2(X_2)\,,
\end{eqnarray*}
which implies
\begin{equation}\label{eq3}
Mw_1(X_1)+w_2(X_2)= Mt_1\,.
\end{equation}
If $w_1(X_1) \le t_1-1$ then, using~\eqref{ineq5}, we obtain
$$Mw_1(X_1)+w_2(X_2)\le
Mt_1-M+w_2(X_2)<Mt_1\,,$$
a contradiction with~\eqref{eq3}.
Therefore $w_1(X_1) \ge t_1$.
Moreover, if $w_1(X_1) \ge t_1+1$, then
$$Mw_1(X_1)+w_2(X_2)\ge Mt_1+M+w_2(X_2)>Mt_1\,,$$
again contradicting~\eqref{eq3}.
We infer that $w_1(X_1) = t_1$ and consequently $X_1\in E_1$.
Equation~\eqref{eq3} together with $w_1(X_1) = t_1$ implies that
$w_2(X_2)= 0$. Since $w_2$ is positive on all $V_2$, it follows that $X_2 = \emptyset$.
Therefore, we have $X = \{z\}\cup X_1\in E$.

Now, suppose that $z\not\in X$.
In this case,
$$Mw_1(V_1)+t_2 = t = w(X) = w(X_1)+w(X_2) = Mw_1(X_1)+w_2(X_2)\,,$$
which implies
\begin{equation}\label{ineq1}
Mw_1(X_1)+w_2(X_2)= Mw_1(V_1)+t_2\,.
\end{equation}
We must have $X_1 = V_1$ since if there exists a vertex $v\in V_1\setminus X_1$, then
we would have
\begin{eqnarray*}
Mw_1(X_1)+w_2(X_2)&\le& Mw_1(V_1)-Mw_1(v)+w_2(X_2)\\
&\le& Mw_1(V_1)-M+w_2(X_2)\\
&<&Mw_1(V_1)+t_2\,,
\end{eqnarray*}
where the last inequality follows from~\eqref{ineq5} and $t_2\ge 0$.
Therefore, equality~\eqref{ineq1} simplifies to
$w_2(X_2)= t_2$, and consequently there exists a hyperedge $X_2\in E_2$.
This implies that $X = V_1\cup X_2\in E$.

For the converse direction, suppose that $X$ is a subset of $V$ such that
$X\in E$.  We need to show that $w(X)= t$.
We again consider two cases depending on whether $z\in X$ or not.
Suppose first that $z\in X$.
Then $X = \{z\}\cup X_1$ for some $X_1\in E_1$.
Due to the property of $w_1$, we have $w_1(X_1)= t_1$.
Consequently,
\begin{eqnarray*}
w(X) &=& w(z)+w(X_1)\\
&=& M(w_1(V_1)-t_1)+t_2 + Mw_1(X_1)\\
&=& Mw_1(V_1)-Mt_1+t_2 + Mt_1\\
&= & Mw_1(V_1)+t_2~=~t\,.
\end{eqnarray*}
Suppose now that $z\not\in X$.
Then $X = V_1\cup X_2$ for some $X_2\in E_2$.
Due to the property of $w_2$, we have $w_2(X_2)= t_2$.
Consequently,
\begin{eqnarray*}
w(X) &=& w(V_1)+w(X_2) \\
&=& Mw_1(V_1)+w_2(X_2)\\
&= & Mw_1(V_1)+t_2~=~t\,.
\end{eqnarray*}
This shows that we have $w(X)= t$ whenever $X\in E$, and completes the proof.
\end{proof}
\end{document}